\documentclass{article}

\usepackage[utf8]{inputenc}
\usepackage{amsmath}
\usepackage{amssymb}
\usepackage{yfonts}
\usepackage{faktor}
\usepackage{amsfonts}
\usepackage{amsthm}
\usepackage{bbding}
\usepackage{bm}
\usepackage{graphicx}
\usepackage{fancyvrb}
\usepackage{xargs}

\title{Minimization Principle for Polynomial-Time Predicates and Forcing in Bounded Arithmetic}
\author{
Mykyta Narusevych\\
Faculty of Mathematics and Physics, Charles University\\
Sokolovsk\'a 83, 186 75 Prague, Czech Republic
}
\date{}

\begin{document}

\newcommand{\T}{\textsf{T}^1_2}
\newcommand{\Talpha}{\textsf{T}^1_2(\alpha)}
\newcommand{\Salpha}{\textsf{S}^1_2(\alpha)}
\newcommand{\Dalpha}{\Delta^b_1(\alpha)}

\newcommand{\Tmin}{\textsf{TMIN}(\alpha)}

\newcommand{\M}{\mathbb{M}}
\newcommand{\N}{\mathbb{N}}
\newcommand{\I}{\mathbb{I}}
\newcommand{\J}{\mathbb{J}}
\newcommand{\calS}{\mathcal{S}}
\newcommand{\bbO}{\mathbb{O}}

\newcommand{\TOUR}{\textsf{TOUR}}
\newcommand{\MIN}{\textsf{MIN}}

\newcommand{\deff}[1]{\textbf{#1}}
\newcommand{\stepmark}[1]{\underline{\textbf{#1}}}
\newcommand{\todo}[1]{\textbf{[TODO: #1]}}

\newcommand{\forcingFrame}{(P, \preceq)}

\newcommand{\trig}{\triangleleft}

\theoremstyle{plain}
\newtheorem{theorem}{Theorem}[section]
\newtheorem{lemma}[theorem]{Lemma}
\newtheorem*{lemma*}{Lemma}
\newtheorem{corollary}[theorem]{Corollary}
\newtheorem{proposition}[theorem]{Proposition}

\theoremstyle{definition}
\newtheorem{definition}[theorem]{Definition}
\newtheorem{question}[theorem]{Question}
\newtheorem{conjecture}[theorem]{Conjecture}
\newtheorem{hypothesis}[theorem]{Hypothesis}
\newtheorem{problem}[theorem]{Problem}
\newtheorem*{example}{Example}

\theoremstyle{remark}
\newtheorem*{remark}{Remark}
\newtheorem*{fact}{Fact}

\maketitle

\begin{abstract}
\sloppy
    We formulate a factorization hypothesis for linear orders definable by oracle polynomial-time machines. Assuming this hypothesis, we provide an approach to constructing a nonstandard model satisfying the minimization scheme for polynomial-time predicates while violating the tournament principle for graphs with polynomial-time edge relation.
\end{abstract}

\section{Introduction}

Unprovability of combinatorial statements in weak fragments of bounded arithmetic is a topic with a long history \cite[Chapter 11]{krajicek95}. While the weak theory $\Talpha$ and some of its extensions come with known unprovability results \cite{riis93,atseriasthapen14,narusevych25}, there are no general methods able to tackle theories above $\textsf{T}^2_2(\alpha)$, except the \textit{switching lemma} \cite{ajtai88} and \textit{lifting theorems} \cite{chiarikrajicek98} (which also use the switching lemma). However, mentioned methods seem difficult to adapt to fine-grained separations between successive theories $\textsf{T}^i_2(\alpha)$ and $\textsf{T}^{i+1}_2(\alpha)$: while those theories are known to be distinct \cite[Section 10.4]{krajicek95}, no formulas of \textit{fixed complexity} are known to separate all of them.

In this paper, we follow on an earlier work of analyzing extensions of $\Talpha$ and identify a computational condition on polynomial-time definable linear orderings, the \textit{Factorization Hypothesis}. Our main result is that, over the base theory $\Talpha$, the \textit{tournament principle} is not implied by the \textit{minimization principle} \cite{chiarikrajicek98} restricted to those polynomial-time linear orders which satisfy this hypothesis. In fact, we conjecture that the Factorization Hypothesis is satisfied by all polynomial-time linear orders.

\section{Preliminaries}

\subsection{Bounded Arithmetic}

We start with a brief overview of bounded arithmetic. For further details, refer to \cite{buss85} and \cite{krajicek95}.

Our primary focus is the relativized bounded arithmetic theory $\Talpha$, with $\alpha$ being a unary predicate symbol. The language of this theory, in addition to $\alpha$, includes the usual arithmetical symbols $0, 1, +, \cdot, \leq$, together with unary operations $\lfloor\frac{\cdot}{2}\rfloor$ and $|\cdot|$, and the binary function symbol $\#$. The symbol $|\cdot|$ denotes bit length, i.e., $|0| = 0$ and, for $x>0$, $|x| = \lfloor \log_{2} x \rfloor + 1$. The symbol $\#$ denotes the binary \textit{smash function}, standardly defined as $2^{|x|\cdot |y|}$. We denote this language by $L_{\textsf{S}_2}(\alpha)$, with $L_{\textsf{S}_2}$ denoting $L_{\textsf{S}_2}(\alpha) \setminus \{\alpha\}$.

A formula is called \deff{sharply bounded} if it is bounded and the bounding terms are of the form $|t|$. The class of sharply bounded formulas is denoted by $\Delta^b_0(\alpha)$.

The class $\Sigma^b_1(\alpha)$ is defined as the smallest set of bounded formulas extending $\Delta^b_0(\alpha)$ and closed under $\land$, $\lor$, sharply bounded quantification, and bounded existential quantification. The class $\Pi^b_1(\alpha)$ is defined as the set of formulas $\varphi$ such that $\neg \varphi \in \Sigma^b_1(\alpha)$, equivalently, as the closure of $\Delta^b_0(\alpha)$ under $\land$, $\lor$, sharply bounded quantification, and bounded universal quantification.

The theory $\Talpha$ is axiomatized by a finite number of true statements expressing the most basic properties of the symbols, such as commutativity of $+$ and $\cdot$, or that $|2x+1| = |x|+1$, forming the so-called \textsf{BASIC} theory; see \cite[Definition 5.2.1]{krajicek95}. This is augmented by the induction axiom scheme for all $\Sigma^b_1(\alpha)$-formulas. Formulas are allowed to contain free variables.

For a formula $\varphi(x, \overline{y})$, denote by $\textsf{PIND}(\varphi)$ the following sentence:
\begin{align*}
    \forall \overline{y} \: [(\varphi(0, \overline{y}) \land \forall x (\varphi(\lfloor \frac{x}{2} \rfloor, \overline{y}) \to \varphi(x, \overline{y}))) \to \forall x \varphi(x, \overline{y})].
\end{align*}

The corresponding axiom scheme is called \deff{polynomial induction}. The theory $\Salpha$ is defined as the already mentioned \textsf{BASIC} theory, not treated in the current paper, augmented by the polynomial induction scheme for $\Sigma^b_1(\alpha)$-formulas. 

A $\Sigma^b_1(\alpha)$-formula $\varphi(\overline{x})$ belongs to the class $\Dalpha$ if there is a $\Pi^b_1(\alpha)$-formula $\psi(\overline{x})$ so that $\Salpha \vdash \forall \overline{x} \: \varphi(\overline{x}) \equiv \psi(\overline{x})$. We note that it is more common to denote the previously defined formula class as $\Dalpha$ \textit{in} $\Salpha$ to emphasize the theory over which the equivalence is established.

\begin{theorem}[{Buss \cite{buss85}, see also \cite[Section 7]{krajicek95}}]
\label{buss}
    A relation $R(\overline{x})$ on numbers is computable in polynomial time with $\alpha$-oracle access if and only if it is definable by a $\Delta^b_1(\alpha)$-formula $\varphi(\overline{x})$.
\end{theorem}

\subsection{Forcing}
\label{forcing definition}

Our treatment of forcing follows \cite{atseriasmuller15}. The reader may consult the mentioned paper for a more general treatment of the subject; here, we focus only on features relevant to the current work.

We fix a countable nonstandard model of true arithmetic in $L_{\textsf{S}_2}$, denoted by $\M$. Throughout this subsection, formulas, terms, and so on are allowed to contain parameters from $\M$, unless stated explicitly. In the next section, we restrict the parameter space.

A \deff{forcing frame} is a partially ordered set $\forcingFrame$ (in general, this poset need not be related to $\M$ in any way). Its elements are called \deff{conditions}. When $p \preceq q$, we say that $p$ \deff{extends} $q$. Two conditions $p, q$ are said to be \deff{compatible}, denoted $p \| q$, if they have a common extension; that is, if there exists $r \preceq p,q$. Otherwise, such conditions are said to be \deff{incompatible}, denoted $p \perp q$.

A binary relation $\Vdash$ between conditions and $L_{\textsf{S}_2}(\alpha)(\M)$-sentences is called a \deff{forcing} if it satisfies the following properties for all sentences $\varphi, \psi$, all formulas $\chi(x)$, and all conditions $p, q$.
\begin{itemize}
    \item $p \Vdash \neg \varphi \:$ if and only if  $\: \forall q \preceq p: \: q \nVdash \varphi$;
    \item $p \Vdash \varphi \land \psi \:$ if and only if  $p \Vdash \varphi$ and $p \Vdash \psi$;
    \item $p \Vdash \varphi \lor \psi \:$ if and only if $\: p \Vdash \neg(\neg \varphi \land  \neg \psi)$;
    \item $p \Vdash \forall x \chi(x) \:$ if and only if, for all $a \in \M$, $\: p \Vdash \chi(a)$;
    \item $p \Vdash \exists x \chi(x) \:$ if and only if $\: p \Vdash \neg \forall x \neg \chi(x)$;
    \item if $p \Vdash \varphi$ and $q \preceq p$, then $q \Vdash \varphi$;
    \item if $\forall q \preceq p \: \exists r \preceq q \: r \Vdash \varphi$, then $p \Vdash \varphi$.
\end{itemize}

A forcing $\Vdash$ is called \deff{conservative} if, for any sentence $\varphi$ not involving the symbol $\alpha$, and any condition $p$, $p \Vdash \varphi$ if and only if $\varphi$ is true in $\M$.

To consider only \textit{reasonable} $\Vdash$, we impose the following additional requirement on every conservative forcing relation considered below. This is essentially a stronger and simpler version of \cite[Lemma 2.18]{atseriasmuller15}: given a condition $p$, and closed terms $s,t$, if $\M \vDash t=s$ and $p \Vdash \alpha(t)$, then $p \Vdash \alpha(s)$.

A subset $G \subseteq P$ is called a \deff{filter} if it satisfies
\begin{itemize}
    \item for any $p, q \in G$, there exists $r \in G$ so that $r \preceq p, q$;
    \item for any $p \preceq q$, if $p \in G$, then $q \in G$.
\end{itemize}
A subset $D \subseteq P$ is called \deff{dense} if any $p \in P$ can be extended by some $q \in D$, and a filter $G$ is called \deff{generic} if it intersects all sets from a selected countable collection of dense subsets of $P$. The full definition of genericity \cite[Definition 2.9]{atseriasmuller15} is not relevant for the current discussion. What matters is that such filters exist \cite[Lemma 2.12]{atseriasmuller15}, and that they satisfy Theorem~\ref{forcing theorem} below.

\begin{theorem}[{Truth Lemma \cite[Theorem 2.19]{atseriasmuller15}, and Forcing Completeness \cite[Corollary 2.20 (a)]{atseriasmuller15}}]
\label{forcing theorem}
    Let $\Vdash$ be a conservative forcing relation, and let $G$ be a generic filter. Then there exists an $L_{\textsf{S}_2}(\alpha)$-expansion of $\M$, denoted by $\M[G]$, satisfying, for any $L_{\textsf{S}_2}(\alpha)(\M)$-sentence $\varphi$,
    \begin{align*}
        \M[G] \vDash \varphi \: \text{ if and only if } \: p \Vdash \varphi \: \text{ for some } p \in G.
    \end{align*}

    Moreover, for any $p \in P$ and any $L_{\textsf{S}_2}(\alpha)(\M)$-sentence $\varphi$,
    \begin{align*}
        p \Vdash \varphi \: \text{ if and only if } \: \M[H] \vDash \varphi \: \text{ for all generic $H$ containing }p.
    \end{align*}
\end{theorem}

\section{$\TOUR$, $\MIN$, and the Transfer Lemma}

Let $n$ be a nonstandard number in $\M$, and let $\I$ be the \textit{small canonical cut} over $n$, i.e., the set of numbers in $\M$ bounded above by parameter-free terms in $n$. As usual, the initial interval $[0, \dots, n-1]$ is denoted by $[n]$.

Recall the \textit{pairing function} $\langle \cdot, \cdot \rangle$, defined as
\[
    \langle x, y \rangle := \frac{(x + y) (x + y + 1)}{2} + y.
\]
This is a bijective mapping between $\N \times \N$ and $\N$, and the same term defines, internally in $\M$, a bijection between $\M \times \M$ and $\M$. With such a function, we can turn the unary $\alpha(\cdot)$ into a binary relation $\alpha(\langle \cdot, \cdot \rangle)$, which we denote by $\alpha$ for simplicity as well.

\begin{definition}[{\cite[12.1]{krajicek95}}]
    A \deff{tournament} is a directed graph $(V, E)$ with exactly one directed edge between any two distinct nodes.

    A set $X \subseteq V$ is said to be \deff{dominating} if, for any vertex $w \in V \setminus X$, there is a $v \in X$ so that $(v, w) \in E$.
\end{definition}

It is a well-known fact that any tournament with $m$ vertices contains a dominating set of size $\leq |m| + 1$; see, e.g., \cite[2.5]{megiddovishkin88}.

The following formula $\TOUR(\alpha, u)$ expresses the above fact formally for graphs with edge-relation defined by $\alpha$, with $u$ being a free variable:
\begin{gather*}
    \exists x, y < u \: (x \neq y) \land (\neg \alpha(x, y) \land \neg \alpha(y, x)) \\
    \lor \\
    \exists x, y < u \: (x \neq y) \land (\alpha(x, y) \land \alpha(y, x)) \\
    \lor \\
    \exists X \subseteq [u] (|X| \leq |u| + 1) \:
    \forall x < u \:
    \bigl(x \in X \lor \exists y < u \: (y \in X \land \alpha(y, x))\bigr).
\end{gather*}

The sets $X$ quantified in the above expression are represented as sequences of length at most $|u| + 1$, or equivalently by padded sequences of length $|u|+1$ together with a length parameter. Thus they can be encoded as numbers from $\M$ of size $O(u^{|u|+1})$.

Let $\forcingFrame$ denote the forcing frame of definable in $\M$ partial orientations of $K_n$, the complete graph on $n$ vertices, of size polynomial in $|n|$. Conditions of such a frame can be identified with sets of ordered pairs of distinct elements of $[n]$, containing at most one of $(a,b)$ and $(b,a)$ for each unordered pair $\{a,b\}$. We then define a conservative forcing relation $\Vdash$ in the natural way: $p \Vdash \alpha(a, b)$ if and only if $(a, b) \in p$, for numbers $a, b \in \M$.

The following can be proved in the exact same way as similar statements in \cite[Section 4]{atseriasmuller15}.

\begin{theorem}
    A generic expansion $\M[G]$ violates $\TOUR(\alpha, n)$ and satisfies the least number principle for all $\Sigma^b_1(\alpha)$-formulas with parameters from $\I$.

    As a corollary, the structure $\I[G]$, defined as the substructure of $\M[G]$ with domain $\I$, is a model of $\Talpha$ violating $\TOUR(\alpha, n)$.
\end{theorem}

\begin{remark}
    It can be shown that $\I[G]$ preserves the \textit{pigeonhole principle} for $\Delta^b_1(\alpha)$-formulas as well; see \cite[Section 5.1]{narusevych25}.
\end{remark}

The main goal of the current work is to present a condition on $\Dalpha$-definable linear orders, stated formally in Hypothesis~\ref{factorization hypothesis}, that implies that $\I[G]$ satisfies the full \textit{minimization scheme} for such orders. Moreover, we conjecture that all $\Dalpha$-definable linear orders obey this condition. As a corollary, our conjecture implies that the tournament principle is not provable from the minimization principle, a statement which currently seems to lie outside the scope of known separation techniques.

\begin{definition}[{\cite{chiarikrajicek98}}]
\label{min definition}
    Given a formula $\varphi(x, y)$, possibly with additional free variables, and a variable $u$ distinct from those of $\varphi$, define $\MIN(\varphi, u)$ as the disjunction of the following five formulas:
    \begin{align*}
        \MIN_0(\varphi,u) &:=
        \exists x < u \: \varphi(x, x),\\
        \MIN_1(\varphi,u) &:=
        \exists x, y < u \: (x \neq y) \land \varphi(x, y) \land \varphi(y ,x),\\
        \MIN_2(\varphi,u) &:=
        \exists x, y < u \: (x \neq y) \land \neg \varphi(x, y) \land \neg \varphi(y, x),\\
        \MIN_3(\varphi,u) &:=
        \exists x, y, z < u \: (x \neq y) \land (y \neq z) \land (x \neq z) \land \varphi(x, y) \land \varphi(y, z) \land \neg \varphi(x, z),\\
        \MIN_4(\varphi,u) &:=
        \exists x < u \: \forall y < u \: (x = y) \lor \varphi(x, y).
    \end{align*}
    Hence,
    \[
        \MIN(\varphi,u) = \bigvee_{i=0}^4 \MIN_i(\varphi,u).
    \]

    The corresponding axiom scheme is called the \deff{minimization scheme}.

    We denote by $\Tmin$ the theory $\Talpha$ augmented by the minimization scheme for all $\Delta^b_1(\alpha)$-formulas.
\end{definition}

\begin{remark}
    It follows from the Riis' theorem \cite{riis93} that $\Tmin$ is strictly stronger than $\Talpha$. In fact, by considering lexicographical ordering on witnesses, it is easy to show that, over a sufficiently strong base theory, for example $\Salpha$, the minimization principle for $\Delta^b_1(\alpha)$-formulas implies the least number principle for $\Sigma^b_1(\alpha)$-formulas.
\end{remark}

To formulate the main conjecture, Hypothesis~\ref{factorization hypothesis}, we begin by directly analyzing the following question. Recall that $\forcingFrame$ refers to the forcing frame of definable partial orientations on $K_n$ with sizes polynomial in $|n|$.

\begin{question}
\label{main question}
    Does $\I[G] \vDash \Tmin$?
\end{question}

Following the standard forcing argumentation, we pick an arbitrary condition $p$, a number $m \in \I$, and a $\Delta^b_1(\alpha)$-formula $\varphi(x, y)$ with parameters from $\I$, with the goal of showing that $p$ can be extended to a condition $q$ forcing $\MIN(\varphi, m)$.

First, recall a fundamental property of $\Vdash$ \cite[Lemma 2.6 (4)]{atseriasmuller15}, stating that $q \nVdash \psi$ if and only if there exists $r \preceq q$ such that $r \Vdash \neg \psi$, for all conditions $q$ and any sentence $\psi$. Thus, we may assume
\[
    p \Vdash \bigwedge_{i=0}^3 \neg \MIN_i(\varphi, m),
\]
i.e., $p$ forces that $\varphi(x, y)$ defines a strict linear ordering on $[m]$. Otherwise, there is an extension $q \preceq p$ forcing the negation of this conjunction, hence forcing one of the first four disjuncts of $\MIN(\varphi,m)$, and we are done.

It is then enough to show that $p \nVdash \neg \MIN_4(\varphi, m)$. Equivalently, by the forcing clauses from Section~\ref{forcing definition}, it is enough to find a number $m_0 \in [m]$ and a condition $q \preceq p$ such that
\[
    q \Vdash \forall x < m \: \bigl((x = m_0) \lor \varphi(m_0, x)\bigr).
\]
A priori, this is not a simple task. Forcing the universal sentence
\[
    \forall x < m \: \bigl((x = m_0) \lor \varphi(m_0, x)\bigr)
\]
is equivalent to forcing all $\Delta^b_1(\alpha)$-sentences $\varphi(m_0, m_1)$ for all $m_1 \neq m_0 < m$ simultaneously. By itself, this might not be possible to accomplish via conditions of small size.

The only way to overcome the above problem is to use, in a non-trivial way, the fact that $\varphi$ \textit{is forced} to define a strict linear ordering.

Below, we treat the cut $\I$ as an $L_{\textsf{S}_2}$-substructure of $\M$.

\begin{definition}
    Let $\mathcal{O}$ be a class of total orientations of $K_n$. Neither the class nor its members are required to be definable in $\M$. For $O \in \mathcal{O}$, we denote by $(\I, \alpha^O)$ the expansion of $\I$ obtained by interpreting $\alpha$ as $O$.

    Given a condition $q$ and an $L_{\textsf{S}_2}(\alpha)$-sentence $\psi$ with parameters from $\I$, we say that $q$ \deff{forces} $\psi$ \deff{with respect to} $\mathcal{O}$, denoted by $q \Vdash^*_{\mathcal{O}} \psi$, if and only if
    \[
        (\I, \alpha^O) \vDash \psi
    \]
    for every $O \in \mathcal{O}$ such that $O \supset q$.
\end{definition}

\begin{remark}
    We deliberately use the notation $\Vdash^*_\cdot$, instead of $\Vdash_\cdot$, to emphasize that, depending on the class $\mathcal{O}$, the relation $\Vdash^*_\mathcal{O}$ may not satisfy the conditions of forcing from Section~\ref{forcing definition}. Nevertheless, the forcing notation is convenient and will be used throughout.
\end{remark}

From this perspective, the forcing completeness theorem, Theorem~\ref{forcing theorem}, equates the standard forcing relation $\Vdash$ with the relative forcing relation $\Vdash^*_\mathcal{O}$ for $\mathcal{O}$ being the class of all orientations induced by generic filters. The other $\mathcal{O}$ of obvious interest is the class of all internally coded total orientations of $K_n$. Equivalently, these are the orientations definable in $\M$ with parameters from $\M$. We denote this class by $\mathcal{D}$.

Now, given some $q$ and a bounded sentence $\psi$, we want to express $q \Vdash^*_\mathcal{D} \psi$ as an $L_{\textsf{S}_2}$-sentence ``$q \Vdash^*_\mathcal{D} \psi$'' with parameters from $\M$, so that $q \Vdash^*_\mathcal{D} \psi$ if and only if $\M \vDash$ ``$q \Vdash^*_\mathcal{D} \psi$''. In more detail, we first apply a Paris--Wilkie style propositional translation to $\psi$; see \cite[Section 10]{krajicek19}. This gives a propositional formula $\langle \psi \rangle$ on $\binom{n}{2}$ variables (while $\alpha$ is binary on $[n]$, since we restrict ourselves only to orientations of $K_n$, it is enough to consider only unordered pairs of distinct elements of $[n]$, hence the size $\binom{n}{2}$). Note that such a formula can be encoded as a number in $\M$. The condition $q$ is then identified with the corresponding partial assignment to those variables denoted as $\tilde{q}$. Since every internally coded total assignment is definable in $\M$ from its code, the fact that $q \Vdash^*_{\mathcal{D}} \psi$ is equivalent to the statement that $\langle \psi \rangle \vert_{\tilde{q}}$ is a tautology. The latter can be expressed in $\M$, which we denote as ``$q \Vdash^*_\mathcal{D} \psi$''. Notice that the sentence ``$q \Vdash^*_\mathcal{D} \psi$'' contains parameters outside of $\I$, as we are quantifying over the set of all total assignments to $\binom{n}{2}$-many variables.

\begin{lemma}[Transfer Lemma]
\label{transfer lemma}
    Let $\psi$ be a $\Pi^b_1(\alpha)$-sentence with parameters from $\I$, and let $q$ be a condition. Then the following are equivalent:
    \begin{enumerate}
        \item $q \Vdash \psi$;
        \item $q \Vdash^*_{\mathcal{D}} \psi$;
        \item  $\M \vDash ``q \Vdash^*_{\mathcal{D}} \psi"$.
    \end{enumerate}
\end{lemma}

\begin{proof}
    The argument rests on the observation that the sentence $\psi$ can be put into a CNF whose clauses have size polynomial in $|n|$. More formally, after expanding sharply bounded quantifiers and applying the Paris--Wilkie translation, $\psi$ can be represented, for the purposes of forcing, in the form
    \begin{align*}
        \forall x < a \exists y < |n|^C \: \theta(x, y),
    \end{align*}
    for some fixed standard constant $C$ and some $a \in \M$, where each $\theta(x,y)$ is an $\alpha$-literal with parameters from $\I$. Concretely, we may assume that there is an $L_{\textsf{S}_2}$-definable function $f(\cdot,\cdot)$ with values in $\I$, and an $L_{\textsf{S}_2}$-formula $\theta^*(x,y)$ not containing $\alpha$, such that
    \[
        \theta(x,y) \equiv
        \bigl(\theta^*(x,y)\land \alpha(f(x,y))\bigr)
        \lor
        \bigl(\neg\theta^*(x,y)\land \neg\alpha(f(x,y))\bigr).
    \]
    The translation is such that $p \Vdash \psi$ if and only if $p \Vdash \psi^*$, and similarly for $\Vdash^*_\mathcal{D}$, for any condition $p$.
    
    Assume $q \nVdash \psi^*$. We find $r \preceq q$ such that
    \[
        r \Vdash \exists x<a \: \forall y < |n|^C \: \neg \theta(x,y).
    \]
    Extending further if necessary, we may fix a particular $b<a$ such that
    \[
        r \Vdash \forall y < |n|^C \: \neg \theta(b, y).
    \]
    By the forcing clauses from Section~\ref{forcing definition}, and the definition of $\theta$, it follows that, for each $c<|n|^C$, the condition $r$ decides the corresponding $\alpha$-literal so as to force $\neg\theta(b,c)$. Altogether, this implies
    \[
        r \Vdash^*_{\mathcal{D}} \forall y < |n|^C \: \neg \theta(b, y),
    \]
    and hence $q \nVdash^*_{\mathcal{D}} \psi^*$.

    Conversely, assume $q \nVdash^*_{\mathcal{D}} \psi^*$. We pick an internally coded total orientation $D$ extending $q$ such that
    \[
        (\I, \alpha^D) \vDash \neg \psi^*.
    \]
    Then there is a particular $b<a$ such that
    \[
        (\I, \alpha^D) \vDash \forall y < |n|^C \: \neg \theta(b, y).
    \]
    Since $D$ is internally coded in $\M$, we can pick the $|n|^C$-sized subset of $D$ already witnessing all oracle answers needed to verify
    \[
        \forall y < |n|^C \: \neg \theta(b, y).
    \]
    Such a subset is a condition from $\forcingFrame$ compatible with $q$, and its union with $q$ clearly forces $\neg \psi^*$.

    The equivalence between points 2 and 3 follows from the discussion preceding Lemma~\ref{transfer lemma}.
\end{proof}

The main utility of the Transfer Lemma, Lemma~\ref{transfer lemma}, is to equate a statement $q \Vdash \psi$ to ``$q \Vdash^*_\mathcal{D} \psi$'', for $\Pi^b_1(\alpha)$-sentences $\psi$, which can then be analyzed inside the ground model $\M$. This is in line with how the least number principle for $\Sigma^b_1(\alpha)$ is forced in \cite{atseriasmuller15}; see the notion of \textit{partial definability} of $\Vdash$ in the mentioned paper. The main difference here is that we identify a class of \textit{globally definable structures}, i.e., totally defined interpretations of $\alpha$, for which the transfer holds, instead of relying on \textit{small partial substructures}, as is usually the case.

\begin{corollary}
\label{line of attack}
    Assume $p$ is a condition such that
    \[
        p \Vdash \bigwedge_{i=0}^3 \neg \MIN_i(\varphi, u).
    \]
    Then,
    \[
        ``p \Vdash^*_\mathcal{D} \bigwedge_{i=0}^3 \neg \MIN_i(\varphi, u)''
    \]
    holds in $\M$.
    
    Furthermore, assume $p$ can be extended to $q$ such that
    \begin{align*}
        \M \vDash ``q \Vdash^*_\mathcal{D} \forall y < m \: (m_0 = y) \lor \varphi(m_0, y)"
    \end{align*}
    for a particular $m_0 < m$. Then $q \Vdash \MIN_4(\varphi, m)$.
\end{corollary}

\section{Factorization Hypothesis}

By Theorem~\ref{buss}, the relation $\varphi(x, y)$ can be decided by a Turing machine $T[\alpha](x, y)$ which queries an oracle $\alpha$, and which runs in time polynomial in the bit lengths of the input numbers. We choose this notation to emphasize that $\alpha$ is not a regular input, but an oracle queried by $T$ during the computation.

Since parameters of $\varphi$ are from $\I$, and $x, y$ are restricted to the interval $[m]$ with $m \in \I$, the runtime of $T[\alpha](x, y)$ is bounded above by a polynomial in $|n|$. Oracles are identified with total assignments to a fixed set of $\binom{n}{2}$ variables. Given such a total assignment $s$, we denote by $T[\alpha_s](x, y)$ the result of substituting $s$ for $\alpha$. Then the condition $p$ is identified with a partial assignment to the same variables, and the class $\mathcal{D}$ is identified with all internally coded total assignments.

The statement
\[
    ``p \Vdash^*_\mathcal{D} \bigwedge_{i=0}^3 \neg \MIN_i(\varphi, m)''
\]
now means that the Turing machine $T[\alpha_s](x, y)$, for all total assignments $s$ agreeing with $p$, defines a strict linear order on $[m]$. Finding $q \preceq p$ satisfying
\[
    ``q \Vdash^*_\mathcal{D} \forall y < m \: (m_0 = y) \lor \varphi(m_0, y)''
\]
amounts to assigning polynomially many additional variables so as to fix the smallest element of the ordering defined by $T[\alpha](x, y)$ on $[m]$.

This leads to the following problem, with $p$ fixed to $\emptyset$ without loss of generality.

\begin{question}
\label{ordering problem}
    Let $n$ be a number, and let $T[\alpha](\overline{x}, x, y)$ be a polynomial-time oracle Turing machine, with oracle being a $0/1$ assignment to the variables encoded as integers from $[\binom{n}{2}]$. A query of $T$ is of the form $a \mapsto ?$ for a number $a \in [\binom{n}{2}]$, with the oracle answering $0$ or $1$ consistently with previous answers.

    Let $\overline{a}$ be a tuple of numbers all polynomial in $|n|$, let $m$ be a number quasi-polynomial in $n$.

    Assume that, for all total assignments $s$, $T[\alpha_s](\overline{a}, x, y)$ defines a strict linear order on $[m]$. Denote by $m_s$ the minimal element of such an ordering.

    Is there a partial assignment $q$ of size polynomial in $|n|$, so that $m_s = m_l$ for any pair of total assignments $s, l$ agreeing with $q$?
\end{question}

\begin{theorem}
\label{one query theorem}
    Question~\ref{ordering problem} has a positive answer for $T$ which can query the oracle at most once during each computation.
\end{theorem}

\begin{proof}
    We prove a slightly stronger statement. Concretely, we show that there is a standard constant $C$ such that, for any subset $X \subseteq [m]$ and any partial assignment $p$, there exists a partial assignment $q$ extending $p$ of size
    \[
        |q| \leq |p| + C\log_2 |X|
    \]
    so that, for all total assignments agreeing with $q$, the orderings defined by $T[\alpha](\overline{a}, x, y)$ and restricted to $X$ share the same minimal element. Here $|X|$ denotes the cardinality of $X$, and $|p|$ denotes the number of variables assigned by $p$. Note that $\log_2|X|$ is polynomial in $|n|$.

    We choose $C$ large enough so that
    \[
        3 \leq C \cdot \log_2 3,
    \]
    and
    \[
        1 + C\log_2\left(\frac{5}{6}N + \frac{1}{2}\right) \leq C\log_2 N
    \]
    for all $N \geq 4$.

    We proceed by complete induction on $|X|$, with the base cases $|X| \in \{0, 1, 2\}$ trivial, and the case $|X| = 3$ following from the assumption $3 \leq C \cdot \log_2 3$.

    Let $x \trig y$ be shorthand for $T[\alpha](\overline{a}, x, y)=1$. We borrow the forcing notation and, given a partial assignment $p$ and a pair of numbers $a,b$, denote $p \Vdash a \trig b$ if and only if, for all total assignments $s$ agreeing with $p$, it holds that
    \[
        T[\alpha_s](\overline{a}, a, b)=1.
    \]

    \stepmark{Step 1.}
    Assume that, for some distinct $a,b\in X$, $p\Vdash a\trig b$. By induction, there exists a number $c\in X\setminus\{b\}$ and an extension $q$ of $p$ of size
    \[
        |q|\leq |p|+C\log_2(|X|-1)
    \]
    so that $q\Vdash c\trig d$ for all $d\in X\setminus\{b\}$ distinct from $c$. It follows that $q\Vdash c\trig b$ as well, since $q\Vdash c\trig a$, $p\Vdash a\trig b$, and $p$ forces transitivity of $\trig$ on the whole set $[m]$. Thus $c$ is forced to be the minimum of $X$.

    For the remainder of the proof, we assume that $p$ does not force $a\trig b$ for any two distinct members of $X$.

    \stepmark{Step 2.}
    For distinct $a,b\in X$, by the previous discussion, there exists a single variable that $T$ queries so as to compute the output on $a,b$. Denote such a variable by $v_{a,b}$. Since $p$ does not decide either direction between $a$ and $b$, both computations on $(a,b)$ and $(b,a)$ must depend on their unique oracle query. Moreover, because $p$ forces exactly one of $a\trig b$ and $b\trig a$, these two queried variables must coincide. Otherwise, one could choose the two oracle answers independently and make the two computations agree, contradicting the fact that $p$ forces a strict linear order. Hence $v_{a,b}=v_{b,a}$, and
    \begin{align*}
        p \cup \{v_{a, b} = 1\} \Vdash (a \trig b)
        \quad \text{and} \quad
        p \cup \{v_{a, b} = 0\} \Vdash (b \trig a),
    \end{align*}
    or
    \begin{align*}
        p \cup \{v_{a, b} = 1\} \Vdash (b \trig a)
        \quad \text{and} \quad
        p \cup \{v_{a, b} = 0\} \Vdash (a \trig b).
    \end{align*}

    \stepmark{Step 3.}
    \begin{lemma*}[Triangle Lemma]
    \label{triangle lemma}
        For any triplet of distinct $a,b,c\in X$, it holds that
        \[
            |\{v_{a,b}, v_{a,c}, v_{b,c}\}| \leq 2.
        \]
    \end{lemma*}

    Suppose otherwise that the three variables $v_{a,b}$, $v_{a,c}$, and $v_{b,c}$ were distinct. Since none of the pairwise comparisons is decided by $p$, and since each comparison depends on its corresponding variable, we could assign these three variables independently so as to force
    \[
        a\trig b,\quad b\trig c,\quad c\trig a.
    \]
    This contradicts the fact that every total assignment extending $p$ defines a strict linear order. This proves the Triangle Lemma.

    \stepmark{Step 4.}
    Fix an arbitrary $a\in X$, and denote $v_{a,b}$ simply by $v_b$ for $b\in X\setminus\{a\}$. We define the following binary relation $\sim$ on $X\setminus\{a\}$:
    \[
        b\sim c \quad \text{if and only if} \quad v_b=v_c.
    \]
    Since $\sim$ is an equivalence relation, it induces a partition
    \[
        X\setminus\{a\}=\bigsqcup_i X_i.
    \]

    \stepmark{Step 5.}
    Assume that some component $X_i$ has cardinality at least $\frac{1}{3}|X|$. Let $v$ denote $v_b$ for an arbitrary $b\in X_i$. By averaging, there is $s\in\{0,1\}$ so that, for at least half of the elements $c\in X_i$,
    \[
        p\cup\{v=s\}\Vdash a\trig c.
    \]
    Denote the set of such elements $c$ by $X'$. It follows that
    \[
        |X'|\geq \frac{1}{6}|X|.
    \]
    By induction, we can extend $p\cup\{v=s\}$ to a condition $q$ of size
    \[
        |q|
        \leq
        |p|+1+C\log_2\left(\frac{5}{6}|X|\right)
        \leq
        |p|+C\log_2|X|
    \]
    which forces the $\trig$-minimality of some $d\in X\setminus X'$. It follows that $q\Vdash d\trig a$, or $d=a$, and hence such $d$ is $\trig$-minimal on the whole set $X$ as well.

    For the remainder of the proof, we assume that all of the components $X_i$ have cardinality less than $\frac{1}{3}|X|$.

    \stepmark{Step 6.}
    Given distinct $X_i,X_j$, and arbitrary $b\in X_i$, $c\in X_j$, the Triangle Lemma implies
    \[
        v_{b,c}\in \{v_b,v_c\}.
    \]
    We then define the following grid $G$, indexed by ordered pairs of elements of $X\setminus\{a\}$:
    \[
        G_{b,c} =
        \begin{cases}
        *, & b \sim c;
        \\
        0, & v_{b,c} = v_b (= v_{a,b});
        \\
        1, & v_{b,c} = v_c (= v_{a,c}).
        \end{cases}
    \]
    Denote by $G_b$ and $G^c$ the row and column of $G$ corresponding to $b$ and $c$, respectively.

    \stepmark{Step 7.}
    Assume that some row $G_b$ contains at least $\frac{1}{3}|X|$ many $0$s. Let $X'$ denote the set of all $c$ such that $G_{b,c}=0$. For any $c\in X'$, it holds that $v_{b,c}=v_b$. We can then proceed similarly to Step 5: find an $s\in\{0,1\}$ such that $p\cup\{v_b=s\}$ forces $b\trig c$ for at least half of the elements of $X'$, and apply induction to $X\setminus X'$, whose size is at most $\frac{5}{6}|X|$.

    For the remainder of the proof, we assume that every row of $G$ contains fewer than $\frac{1}{3}|X|$ many $0$s.

    \stepmark{Step 8.}
    Note that each row and each column of $G$ contains fewer than $\frac{1}{3}|X|$ many $*$s. Let $o$ denote the maximum number of $1$s contained in a single column of $G$. By double-counting the number of $1$s in $G$, and using the assumptions that every row contains fewer than $\frac{1}{3}|X|$ many $0$s and fewer than $\frac{1}{3}|X|$ many $*$s, it follows that
    \[
        |X \setminus \{a\}| \cdot o > |X \setminus \{a\}| \cdot (|X \setminus \{a\}| - \frac{2}{3}|X|),
    \]
    which implies $o > \frac{1}{3}|X| - 1$.

    The column version of Step 7 now applies: if a column $G^c$ contains at least $\frac{1}{3}|X| - 1$-many $1$s, then we can fix $v_c$ appropriately to $s$ to force $c\trig b$ for at least $\frac{1}{6}|X| - \frac{1}{2}$-many $b$s. By induction, we can extend $p\cup\{v_c=s\}$ to a condition $q$ of size
    \[
        |q|
        \leq
        |p|+1+C\log_2\left(\frac{5}{6}|X| + \frac{1}{2}\right)
        \leq
        |p|+C\log_2|X|,
    \]
    finishing the proof.
\end{proof}

Unfortunately, we are unable to extend the above proof method to the general case needed for resolving Question~\ref{ordering problem}. We believe that the crucial missing detail is a more general analogue of the Triangle Lemma.

Below, we formulate a stronger, but arguably more natural, condition on orderings defined by polynomial-time oracle Turing machines that is enough to resolve Question~\ref{ordering problem}. This condition may be of independent interest.

\begin{definition}
    Let $T[\alpha](\overline{x}, x, y)$, $n$, $m$, and $\overline{a}$ be as in Question~\ref{ordering problem}. We say that $T[\alpha](\overline{a}, x, y)$ \deff{factorizes on} $[m]$ if and only if there is a polynomial-time oracle Turing machine $f[\alpha](\overline{x}, x)$ satisfying the following properties:
    \begin{itemize}
        \item there is a single number $u$ quasi-polynomial in $n$ such that
        \[
            f[\alpha_s](\overline{a}, b) < u
        \]
        for all total assignments $s$ and all $b \in [m]$;
        \item for any total assignment $s$, the function
        \[
            b \mapsto f[\alpha_s](\overline{a}, b)
        \]
        is injective on $[m]$;
        \item for any distinct numbers $a,b\in[m]$, and any total assignment $s$,
        \[
            T[\alpha_s](\overline{a}, a, b) = 1
        \]
        if and only if
        \[
            f[\alpha_s](\overline{a}, a) < f[\alpha_s](\overline{a}, b).
        \]
    \end{itemize}
\end{definition}

We choose the name because the computation of $T[\alpha](\overline{a},a,b)$ \textit{factorizes} into two independent computations $f[\alpha](\overline{a},a)$ and $f[\alpha](\overline{a},b)$.

\begin{proposition}
    Question~\ref{ordering problem} has a positive answer assuming that $T$ factorizes on $[m]$.
\end{proposition}

\begin{proof}
    Since the image of $f[\alpha](\overline{a}, x)$ when varying $\alpha_s$ and $b < m$ is bounded by $u$, it contains the smallest element $m_0$. It is then enough to pick condition $q \Vdash f[\alpha](\overline{a}, b) = m_0$ for some $b$.
\end{proof}

\begin{corollary}
    $\I[G]$ satisfies $\forall u \: \MIN(\varphi,u)$ for all $\Dalpha(\I)$-formulas $\varphi$ which factorize in $\M$.
\end{corollary}

\begin{hypothesis}[Factorization Hypothesis]
\label{factorization hypothesis}
    Any $T[\alpha](\overline{x}, x, y)$ as in Question~\ref{ordering problem} factorizes.
\end{hypothesis}

It is important to keep in mind that $f$ does not need to calculate the exact index of the corresponding number in the given ordering. In other words, $f$ need not be surjective on $[m]$.
 
Below, we give an example of $T$ that factorizes, but for which no polynomial-time indexing function can be constructed.

\begin{example}
    We define $T[\alpha](x,y)$ for $m=\binom{n}{2}$ as follows. Given distinct $a,b\in[m]$, the machine $T[\alpha](a,b)$ queries the variables encoded by $a$ and $b$. If both queries result in the same answer, the numbers are ordered by the usual numerical ordering. Otherwise, the number with query answer $0$ is treated as the smaller.

    To see that such $T$ factorizes, we construct $f$ as follows:
    \[
        f[\alpha](a) =
        \begin{cases}
        a, & \text{if the query } a \mapsto ? \text{ is answered } 0;
        \\
        a + m, & \text{if the query } a \mapsto ? \text{ is answered } 1.
        \end{cases}
    \]

    To see that no polynomial-time oracle machine can compute the exact index of $a$ in the induced ordering, note that the exact rank of $a$ depends on the number of indices whose oracle answers are $0$ or $1$. For $a$ near $\lfloor m/2\rfloor$, changing any one of linearly many still-unqueried oracle bits can change this rank. Since $m=\binom n2$ is much larger than any polynomial in $|n|$, no algorithm making only polynomially many queries in $|n|$ can determine the exact rank uniformly.
\end{example}

\bibliographystyle{plain}
\bibliography{main}

\end{document}